\documentclass[12pt,reqno]{amsart}
\usepackage{amsmath,amssymb,amsthm,amsfonts,graphicx,color}
\usepackage[colorlinks=true]{hyperref}
\usepackage{textcomp}
\usepackage{yfonts}
\usepackage{amssymb}
\usepackage[hmargin=3.2cm, vmargin=3.2cm]{geometry}
\usepackage{dsfont}
\usepackage{upgreek}
\usepackage{mathrsfs}
 \newtheorem{corollary}{\textbf{Corollary}}[section]

\newtheorem{theorem}{\textbf{Theorem}}[section]
\newtheorem{lemma}{\textbf{Lemma}}[section]
\newtheorem{proposition}{\textbf{Proposition}}[section]

 \numberwithin{equation}{section}
\begin{document}

\title{Relative Nash-type and $L^2$-Sobolev inequalities in Dunkl Setting}
%\subtitle{Dunkl  Nash-type  inequalities}


\author{Sami Mustapha   and Mohamed Sifi }


\address{S. Mustapha
              Sorbonne Universit\'e, Tour 25, 5e \'{e}tage
Boite 247, 4 place Jussieu F-75252 Paris Cedex 05 }
              \email{sam@math.jussieu.fr}           %  \\
           \address{ M. Sifi
              Universit\'e de Tunis El Manar,
Facult\'e des Sciences de Tunis, Laboratoire d'Analyse
Math\'ematique et Applications LR11ES11, 2092, Tunis, Tunisia.}
 \email{mohamed.sifi@fst.utm.tn}


\renewcommand{\thefootnote}{}
\maketitle

\begin{abstract}
We investigate local variants of Nash  inequalities in the context of Dunkl operators. Pseudo-Poincar\'e inequalities are first established using
pointwise gradient estimates of the Dunkl heat kernel. These inequalities allow  to obtain  relative Nash-type inequalities  which are used to derive mean value inequalities for subsolutions of the heat equation  on orbits of balls not necessarily centered on the origin.

\end{abstract}
\footnote{MSC  Primary 35A23,  Secondary 26D10 \and 35J05 \and 35K08. }
\section{Introduction and main results}

The Nash inequality \cite{Na} states the existence of  a constant $C>0$,
such that
\begin{equation}\label{1.1}
  \|f\|_{2}^{1+\frac{2}{ N}}\leq C  \|\nabla f\|_{2} \|f\|_{1}^{\frac{2}{N}},
\end{equation}
for  any  $ f\in \mathcal{C}_0^\infty( \mathbb{R}^N )$. This inequality
is  of fundamental importance
because it accounts in a very simple interpolative way
how a control of the $L^2$-norm  of a function, under a  normalization  condition,  results in a lower bound of the  $L^2$-norm of its gradient.
It was introduced by Nash in 1958 to obtain regularity properties of the solutions to parabolic partial differential equations.

Inequality  (\ref{1.1})
generalizes  to the context of Dunkl operators in the following form \cite{V2}
\begin{equation}\label{1.2}
  \|f\|_{2, \kappa }^{1+\frac{2}{ N+2\gamma}}\leq C  \|\nabla_\kappa f\|_{2,\kappa} \|f\|_{1,\kappa}^{\frac{2}{N+2\gamma}} ,\quad f\in \mathcal{C}_0^\infty(\mathbb{R}^N).
\end{equation}
The number $N+2\gamma$ is  the homogeneous dimension
 and $\nabla_\kappa$ is the Dunkl gradient built from the Dunkl
operators.
The  norms $||.||_{k,q}$ are computed with respect to the weighted measure
\begin{equation}\label{1.3}
d\mu_\kappa(x)=\omega_\kappa(x)dx=\prod_{\alpha\in {\mathcal R}_+}|<\alpha,x>|^{2\kappa_\alpha}dx,\end{equation}
where ${\mathcal{R}_+}$ is a fixed positive root system and
$\kappa$ is a nonnegative multiplicity function $\alpha \rightarrow
\kappa_\alpha$ defined on $\mathcal{R}_+$ (see Sect.~\ref{sec:2}). The weight $\omega_\kappa$
is homogeneous of degree $2\gamma$.
For  $\kappa=0$, Dunkl operators reduce to the usual  partial derivatives and $d\mu_0(x)$ is the  Lebesgue measure.
\\
In a recent paper \cite{V2}  Velicu established (\ref{1.2})  and used  it to obtain an elementary proof of
the Sobolev inequality in the cas $p=2$.
 Nash's inequality can be seen as a weaker form of the Sobolev inequality, since  it can be deduced from it using H\"older's inequality, but in fact these two inequalities  are equivalent to each other and equivalent to the ultracontractive bound on the Markov semigroup associated to the Dunkl Laplacian $\Delta_\kappa$ (see \cite{V2}).

The aim of this article is to  investigate  scale-invariant local
variants of Nash-Dunkl inequality
(\ref{1.2}).
Our main result is the following family of Nash-Dunkl inequalities
on balls.

%Recall that the Dunkl Sobolev space
 %$W_\kappa^{1,p}(\mathbb{R}^N)$ is defined  for all $p\geq 1$  as the  space
% of  all functions $f \in L^p(\mathbb{R}^N, d\mu_\kappa)$ for which
% $\nabla_\kappa f $  exists in a weak sens and
% $\nabla_\kappa f \in L^p(\mathbb{R}^N, d\mu_\kappa)$.

\begin{theorem}\label{thm1} \textit{(Relative Nash-Dunkl inequality).}
Let $B\subset \mathbb{R}^N$ be an Euclidean ball of radius $r(B)>0$. Then for any
 $p>1$  there exists a constant $C>0$ independent of $B$
 such that
for  any function $ f\in  \mathcal{C}_0^\infty(B)$,
\begin{equation}\label{1.4}
  \left(\int_B|f|^p d \mu_\kappa \right)^{1+\frac{p'}{N+2\gamma}}\leq C\frac{r(B)^p}{v_\kappa(B)^{\frac{p}{N+2\gamma}}}\left[\int \left(|\nabla_\kappa f|^p+\frac{|f|^p}{r(B)^p}\right)d\mu_\kappa\right]\|f\|_{1, \kappa}^{\frac{p p'}{N+2\gamma}},
\end{equation}
where $p'$ denotes the H\"older conjugate exponent of $p$ and where the volume $v_\kappa(B)$ is computed with respect to the Dunkl measure (\ref{1.3}).
\end{theorem}

For $p=1$, the previous  inequality loses its meaning.  A substitute is given by
the following weak Nash-type inequality.

\begin{theorem}\label{thm2} \textit{(Weak relative Nash-Dunkl inequality).}
Let $B\subset \mathbb{R}^N$ be an Euclidean ball of radius $r(B)>0$. Then, there exists a constant $C>0$ independent of $B$
 such that
for  any function $ f\in  \mathcal{C}_0^\infty(B)$
and $\lambda >0$,

\begin{equation}\label{1.5} \lambda^{1+\frac{1}{N + 2 \gamma}}
  \mu_\kappa \{ |f|\geq \lambda\} \leq
  \frac{Cr(B)}{ v_\kappa (B)^{\frac{1}{N + 2 \gamma}} }
  \left[\int \left(|\nabla_\kappa f|+\frac{|f|}{r(B)}\right)d\mu_\kappa\right]
\|f\|_{1, \kappa}^{\frac{1}{N + 2 \gamma}}.
\end{equation}
\end{theorem}

It is well known (see \cite{CKS}, \cite{Sa0}) that for $p=2$ inequality (\ref{1.4}) is equivalent to the following Sobolev-type inequality.
\begin{theorem}\label{thm3} \textit{(Relative Sobolev-Dunkl inequality).}
Let $B\subset \mathbb{R}^N$ be an Euclidean ball of radius $r(B)>0$. Then
 there exists a constant $C>0$ independent of $B$
 such that
for  any function $ f\in \mathcal{C}_0^\infty(B)$,
\begin{equation}\label{1.6}
  \left(\int_B |f|^{\frac{2(N+ 2\gamma)}{N+2\gamma-2}}d\mu_\kappa\right)^{\frac{N+2\gamma-2}{N+2\gamma}}\leq C\frac{r(B)^2}{v_\kappa(B)^{\frac{2}{N+2\gamma}}}\int\left(|\nabla_\kappa f|^2+\frac{|f|^2}{r(B)^2}\right)d\mu_\kappa.
\end{equation}

\end{theorem}

We qualify these inequalities as relative to refer to the ball $B$ where they are considered. The important point is their invariance by scaling and the fact that the constant $C$ is independent of the ball $B$.

Notice that letting $p=2$ and $r(B)\rightarrow\infty$ in (\ref{1.4}) (resp. in (\ref{1.6})) yields (\ref{1.2})
(resp.  the Dunkl-Sobolev inequality). This results follow from the fact that the Dunkl volume $v_\kappa(B)$ satisfies the lower bound (see Sect.~\ref{sec:2}):
$$v_\kappa(B)\geq c r(B)^{N+2\gamma}.$$

Our method therefore offers an alternative approach
to establishing Nash and Sobolev inequalities in the
context of Dunkl operators allowing to generalize some results of
 \cite{V2}.\\

The inequality of Nash (\ref{1.2}) is easily demonstrated by an adaptation of the original approach of \cite{Na} thanks to  the Dunkl transform.
As Nash points out in \cite{Na}, this inequality was in fact demonstrated, at his request, by E. Stein  and the proof is based on the use of
the Fourier transform.  \\

The local variant (\ref{1.4}) is more difficult to establish. One possible approach is to use pseudo-Poincar\'e  inequalities.
Such Inequalities  were established
by  S.  Adhikari, V. P. Anoop and S.  Parui  in \cite{AAP} for  $p=2$
by Velicu in \cite{V2} for $ 1\leq p \leq 2 $.
The proofs developed in \cite{AAP} and \cite{V2} are different in nature.
The $ L^2 $ nature of the inequality established in \cite{AAP} allows the Dunkl transform to be used  and the Velicu result is based on the use of the carr\'e-du-champ operator and semi-group techniques. \\

The main contribution of this paper is to note that the gradient heat kernel  estimates recently established by Anker et
al in \cite{An2}  allow  to remove
the restrictive hypothesis $1\leq p \leq 2$ and  to derive  pseudo-Poincar\'e inequalities
for all $ p \geq 1 $. Once we have these inequalities we  are able to
adapt the general approach
developed in \cite{BCLS}, \cite{Sa0}, \cite{Sa1} and derive the relative Nash and Sobolev inequalities (\ref{1.4}), (\ref{1.5}) and (\ref{1.6}).
Finally, as an application of the $L^2$-Sobolev inequality (\ref{1.6})
we derive mean value inequalities for subsolutions of the heat equation
using Moser's iteration argument.

\section{Background and preliminaries}\label{sec:2}
In this section we recall some important properties of Dunkl operators
and collect some  preliminary assertions which are necessary in the
proof of our main results. For more
details see   \cite{An}, \cite{deJ}, \cite{Du}, \cite{Du2} and \cite{R1} for
an overview of Dunkl theory.

Throughout the remainder of this paper, $C$ will denote a positive constant which can differ  from one occurrence to another, even
in the same formula and we will use $ A \approx  B$  to say that
the ratio $A/B$ is bounded between two positive constants.

We consider $\mathbb{R}^N$ with the Euclidean scalar product $<.,.>$
and its associated norm $|x|=\sqrt{<x,x>}$. For $\alpha \in
\mathbb{R}^N\setminus \{0\}$, the reflection $\sigma_\alpha$ with
respect to the hyperplan $H_\alpha$ orthogonal to $\alpha$ is given
by
$$ \sigma_\alpha(x) = x -2\frac{<x,\alpha>}{|\alpha|^2}\alpha,
\quad x \in \mathbb{R}^N.$$ A finite set ${\mathcal R} \subset
\mathbb{R}^N \setminus \{0\}$ is called a reduced root subsystem if
${\mathcal R} \cap \mathbb{R} \alpha =\{\mp \alpha\}$ and
$\sigma_\alpha {\mathcal R} = {\mathcal R}$ for all $ \alpha \in
{\mathcal R}$.
The finite group $G$ generated by the reflections $\sigma_\alpha$, $
\alpha \in {\mathcal R}$  is called the Coxeter-Weyl group of $
{\mathcal R}$.

Then,  we fix a $G$-invariant function $\kappa :
{\mathcal R} \longrightarrow \mathbb{C}$ called the multiplicity
function of the root system. We assume in this article that $\kappa$
takes its values in $[0,+\infty[$ and that the root system is
reduced and normalized so that $|\alpha|^2=2$, $ \alpha \in
{\mathcal R}$.

The Dunkl operators $T_j$ $(j=1, \ldots, N)$, introduced
in \cite{Du}, are the following $\kappa$-deformations of the usual
directional derivatives $\partial/{\partial x_j}$ by reflections
$$ T_j f(x) =
\frac{ \partial f}{\partial x_j }(x) + \sum_{\alpha \in {\mathcal
R}^+} \kappa(\alpha) \alpha_j\frac{f(x) -f(\sigma_\alpha x)} {
<x,\alpha>},$$ where ${\mathcal R}^+$ is a positive subsystem of
${\mathcal R}$. The definition is of course independent of the
choice of the positive subsystem since $\kappa$ is $G$-invariant.
The Dunkl operators  $T_j$  are skew-symmetric with respect to the $G$-invariant measure
$$d\mu_\kappa(x)=\omega_\kappa(x)dx=\prod_{\alpha\in \mathcal{R}_+}|<\alpha,x>|^{2k_\alpha}dx.$$
 A
 fundamental property of these differential-difference operators is
 their commutativity, that is to say $T_{k}T_{l}=T_{l}T_{k}$. Closely related to them is the so-called intertwining
 operator $V_{\kappa}$  which is the unique linear
 isomorphism of $\bigoplus_{n\geq 0} \mathcal{P}_{n}$ such that $$
 V_{\kappa}( \mathcal{P}_{n})=\mathcal{P}_{n}, \ \ \ \   V_{\kappa}(1)=1, \ \
 \ T_{j} V_{\kappa}=V_{\kappa}\partial_{j}, \ for \ j=1,...,N,$$
 with $\mathcal{P}_{n}$ the subspace of homogeneous polynomials of
 degree $n$ in $N$ variables.
Even if the positivity of the
 intertwining operator has been established  by M.
 R\"{o}sler, an explicit formula of $V_{\kappa}$ is not known in general.
 However,  the operator $V_{\kappa}$ possesses the integral representation
 $$V_{\kappa}f(x)=\int_{\mathbb{R}^{N}}f(y) d\mu_{x}(y) ,$$
 where
 $\mu_{x}$ is a probability measure on $\mathbb{R}^{N}$ with
 support in the Euclidean ball  of center 0 and radius
 $|x|$.  The function
 $E(x,y)=V_{\kappa}^{x}[e^{<x,y>}]$, where the superscript means
 that $V_{\kappa}$ is applied to the $x$ variable, plays an
 important role in the development of the Dunkl transform.
 In particular, the function  $$E(x,iy)=V_{\kappa}^{x}[e^{i<x,y>}], \ \ \ \ \ \ x, \ y\in \mathbb{R}^{N},$$
 plays the role of $e^{i<x,y>}$ in the ordinary Fourier analysis.
 The Dunkl transform is defined in terms of it by
 $$\mathcal{F}(f)(y)=c_{\kappa}\int_{\mathbb{R}^{d}}f(x)E(x,-iy)d\mu_{\kappa}(x)
dx, \ \  \ \ \ \ \ y\in \mathbb{R}^{N}. $$ If $\kappa=0$, then
$V_{\kappa}=id$ and the Dunkl transform coincides with the usual
Fourier transform. As in the classical case, the Dunkl transform defines a topological automorphism of
$\mathcal{S}(\mathbb{R}^{N})$ and
extends to an isometry of $L^{2}(\mathbb{R}^{N},d\mu_{\kappa}).$

Let $\gamma=\displaystyle \sum_{\alpha\in \mathcal{R}_+}\kappa(\alpha)$. The number $N+2\gamma$ is called the homogeneous dimension, because of the obvious scaling property
\begin{equation}\label{2.1}
  V_\kappa(ta,tr)=t^{N+2\gamma}V_\kappa(a,r),\quad t>0,\end{equation}
where $V_\kappa(a,r)=\mu_\kappa(B_r(a))$,
$B_r(a)$ being the Euclidean  ball of radius $r$ and centered at $a\in \mathbb{R}^N$.

We will also need to use
 the distance $d(a,x) = \min_{\sigma \in G}|x -\sigma a |$ (the distance between the
$G$-orbits $\mathcal{O}(a)$ and $\mathcal{O}(x)$). Obviously, the corresponding balls
$$ B_r^G(a) = \{x \in \mathbb{R}^N, \quad d(a,x) <r\} = \mathcal{O}\left(B_r(a)\right),
 \quad a\in \mathbb{R}^N,\quad  r>0,$$
satisfy
\begin{equation}\label{2.2}
V_\kappa(a,r) \leq \mu_\kappa \left(B_r^G(a)\right) \leq |G|V_\kappa(a,r).
\end{equation}
We will denote by $\nabla_\kappa=(T_1,\ldots, T_N)$  the Dunkl gradient and  $ \Delta_\kappa = \displaystyle \sum_{j=1}^N T_j^2$ the Dunkl-Laplacian operator.
The Dunkl-Laplacian acts on ${\mathcal C}^2$-functions as
$$\Delta_\kappa f(x) =  \Delta  f(x) +
2 \sum_{\alpha \in {\mathcal R}^+} \kappa(\alpha) \left( \frac{<
\nabla f(x), \alpha>}{<\alpha, x>} - \frac{f(x) -f(\sigma_\alpha x)}
{ <x,\alpha>^2} \right),$$ where $\Delta$ is the classical Laplacian
operator on $\mathbb{R}^N$ and $\nabla$ the associated gradient.
The operator $\Delta_\kappa$  is essentially self-adjoint on
$L^2(\mathbb{R}^N, d\mu_\kappa) $ and generates the heat semigroup
$T_t=e^{-t \Delta_\kappa}$, $(t>0)$.
 Via the Dunkl transform, the heat semigroup  is given by
$$ T_t f(x)= \mathcal{F}^{-1}\left(e^{-t|\xi|^2} \mathcal{F}(\xi)\right)(x).$$
Alternately  \cite{RV}
\begin{equation}\label{2.3}
 T_tf(x) = f \star h_t(x) = \int_{\mathbb{R}^N}h_t(x,y)f(y)d\mu_\kappa(y),\quad
t>0,\quad x \in \mathbb{R}^N,
\end{equation}
where the heat kernel $h_t(x,y)$ is given by a smooth positive radial convolution kernel.
Notice that  (\ref{2.3})
defines a strongly continuous semigroup of linear contractions on
$L^p(\mathbb{R}^N, d\mu_\kappa) $, for every $1\leq  p < \infty$.

The heat kernel  satisfies the following  Gaussian upper  bounds
for the orbit distance $ d(x,y)$ (see \cite{An2}):
\begin{equation}\label{2.4}
 h_t(x,y) \leq \frac{C}{\max\{V_\kappa(x, \sqrt t), V_\kappa(y, \sqrt t)\}}
 e^{-\frac{d(x,y)^2}{Ct}},\quad
t>0,\quad x, y \in \mathbb{R}^N.\end{equation}
Another estimate which plays a fundamental role in our analysis is the Gaussian upper estimate of the spatial gradient of the heat kernel \cite{An2}:
\begin{equation}\label{2.5}
 |\nabla_{\kappa, x} h_t(x,y)|  \leq \frac{C}{
\sqrt t \max\{V_\kappa(x, \sqrt t), V_\kappa(y, \sqrt t)\}}
 e^{-\frac{d(x,y)^2}{Ct}},\quad
t>0,\quad x, y \in \mathbb{R}^N.\end{equation}

Finally, the following volume estimates will be important in all subsequent proofs.
The first two assertions are well known [4, Sect. 3].
We include detailed proofs for the reader's convenience.

\begin{proposition}\label{prop1}
\begin{itemize}\item [i)] Let $a\in \mathbb{R}^N$ and $r>0$. Then
\begin{equation}\label{2.6}
 V_\kappa(a,r) \approx
     r^N\prod_{\alpha\in \mathcal{R}_+}(|<\alpha,a>|+r)^{2\kappa(\alpha)}.\end{equation}
\item [ii)] There exists a constant $C\geq 1$ such that for every $a\in \mathbb{R}^N$ and for every $r\geq s>0$,
\begin{equation}\label{2.7}
    C^{-1}\left(\frac{r}{s}\right)^N\leq\frac{V_\kappa(a,r)}{V_\kappa(a,s)}\leq C \left(\frac{r}{s}\right)^{N+2\gamma}.
\end{equation}
\item [iii)] There exists a constant $C\geq 1$ such that for every
 $a\in \mathbb{R}^N$,   $0<s \leq r $ and
  $y \in B_r(a)$,
\begin{equation}\label{2.8}
  V_\kappa (a,r) s^{N+ 2\gamma} \leq  C V_\kappa (y,s) r^{N+ 2\gamma}.
\end{equation}
\end{itemize}
\end{proposition}
\begin{proof} To prove (\ref{2.6}) fix  $a\in \mathbb{R}^N$ and $r>0$, then using the change of variable $x=a+tu$, $0<t<r$ and $u\in S^{N-1}$, we obtain that the volume
$ V_\kappa(a,r)$ is equal to
\begin{eqnarray}\label{2.9} \int_0^r\int_{S^{N-1}}\prod_{\alpha\in \mathcal{R}_+}\left|<\alpha,a>+t<\alpha,u>\right|^{2\kappa(\alpha)}t^{N-1}
dtd\sigma(u),\end{eqnarray}
where $d\sigma$ is the induced Euclidean  measure on the unit sphere $S^{N-1}$.   Thus using the elementary
estimate $|<\alpha,a>|+\sqrt{2}r\leq \sqrt{2}{\left(|<\alpha,a>|+r\right)}$,  we obtain
\begin{eqnarray}\label{2.10}V_\kappa(a,r)  &\leq &
C r^N\prod_{\alpha\in \mathcal{R}_+}(|<\alpha,a>|+r)^{2\kappa(\alpha)}.
\end{eqnarray}

Let now establish a similar lower estimate for $V_\kappa(a,r)$. Using the fact that $\mathcal{R}$ is invariant with respect to the action of the Weyl group $G$, we obtain
\begin{equation}\label{2.11}
    V_\kappa(ga,r)=V_\kappa(a,r),\quad g\in G.
\end{equation}
Property (\ref{2.11}) combined with the scaling property (\ref{2.1}) show that it suffices to estimate $V_\kappa(a,1)$. Let $\mathcal{C}=\left\{x\in \mathbb{R}^N  \,\,:\; <\alpha,x>\, >0, \,\, \alpha\in \mathcal{R}_+ \right\}$
denote the positive Weyl chamber associated to the root system $\mathcal{R}$.  According to (\ref{2.11}) we can suppose that $a\in\overline{\mathcal{C}}$. So by
 (\ref{2.9}) we obtain
$$V_\kappa(a,1)\geq \int_0^1\int_{\Sigma}\prod_{\alpha\in \mathcal{R}_+}\left(<\alpha,a>+t<\alpha,u>\right)^{2\kappa(\alpha)}t^{N-1}dtd\sigma(u),$$
where $\Sigma\subset S^{N-1}$ is chosen such that $\gamma_0=\min\{<\alpha,u>, \alpha\in \mathcal{R}_+, u\in \Sigma\}>0$,   so that
\begin{eqnarray}\label{2.12}
V_\kappa(a,1)&\geq & \int_{\frac{1}{2}}^1\int_{\Sigma}\prod_{\alpha\in \mathcal{R}_+}\left(<\alpha,a>+t\gamma_0 \right)^{2\kappa(\alpha)}t^{N-1}dtd\sigma(u) \nonumber\\ &\geq &
\frac{1}{C}\prod_{\alpha\in \mathcal{R}_+}\left(<\alpha,a>+1 \right)^{2\kappa(\alpha)}.
\end{eqnarray}
Combining  (\ref{2.10}) and (\ref{2.12}) we deduce (\ref{2.6}).
It follows in particular that for $0<s\leq r$, we have
$$ \frac{V_\kappa(a,r)}{V_\kappa(a,s)}\approx \left(\frac{r}{s}\right)^N\prod_{\alpha\in \mathcal{R}_+}\left(\frac{|<\alpha,a>|+r}{|<\alpha,a>|+s}\right)^{2\kappa(\alpha)}. $$
The claim ii) follows from the fact that for $\lambda \geq 0$,
$\lambda + s<\lambda+r< \frac{r}{s}(\lambda +s)$.
Choosing $r=2s$ in (\ref{2.7}) shows that
the measure $\mu_\kappa$ satisfies the doubling property, i.e., there exists a positive constant $C$ such that for every $a\in \mathbb{R}^N$ and for every $r>0$
\begin{equation}\label{2.13}
    V_\kappa(a,2r)\leq CV_\kappa(a,r).
\end{equation}

Finally, let us prove (\ref{2.8}). Using (\ref{2.6}) we see that (\ref{2.8}) is equivalent to
 $$ s^{2\gamma}\prod_{\alpha\in \mathcal{R}_+}(|<\alpha,a>|+r)^{2\kappa(\alpha)}
 \leq C r^{2\gamma}\prod_{\alpha\in \mathcal{R}_+}(|<\alpha,y>|+s)^{2\kappa(\alpha)}$$
 which is to show that
 $$ \prod_{\alpha\in \mathcal{R}_+}\left(\frac{|<\alpha,a>|+r
 }{|<\alpha,y>|+s}\right)^{2\kappa(\alpha)}
 \leq C \left(\frac{r}{s}\right)^{2\gamma}.$$
 The assertion (\ref{2.8}) follows from the fact that $|<\alpha,a>|+s   \leq \sqrt 2 r  +|<\alpha,y>| + s $ and that $r \geq s$.
\end{proof}

\section{ Pseudo-Poincar\'e and Nash-type inequalities}

In this section,
we establish pseudo-Poincar\'e
 inequalities and we make a detailed study
 of   Nash-type inequalities (\ref{1.4}) and (\ref{1.5}). For this purpose we also  get operator
  norm estimates for the family
 $\{T_{t^2},\,\,\,0<t\leq r\}$  acting from
 $ L^1\left( B, d\mu_\kappa\right) \rightarrow
  L^p\left( \mathbb{R}^N, d\mu_\kappa\right)$, for $B\subset  \mathbb{R}^N$,
 $p\geq 1$.
\begin{proposition}\label{prop2} \textit{(Pseudo-Poincar\'{e} inequality).} For any $1\leq p<\infty$, there exists a constant $C > 0$ such that for all  $f\in
\mathcal{C}_0^\infty(\mathbb{R}^N)$, $t>0$, we have
\begin{equation}\label{3.1}
  \|f- T_t f\|_{p,\kappa} \leq  C \sqrt t \|\nabla_\kappa f \|_{p,\kappa}.
\end{equation}
\end{proposition}
\par For the proof of Proposition \ref{prop2} we need the following lemma
\begin{lemma}\label{lem} \textit{(Schur's test) }
Assume that $k$ is a measurable function on $\mathbb{R}^{N}$ that
satisfies the mixed-norm conditions:$$ C_{1}=\sup_{x\in
\mathbb{R}^{N}}\int|k(x,y)|d\mu_{\kappa}(y)<\infty, \ \ \
C_{2}=\sup_{y\in
\mathbb{R}^{N}}\int|k(x,y)|d\mu_{\kappa}(x)<\infty.$$ Then the
integral operator induced by the kernel $k(x,y)$ (i.e. the operator
defined by $\displaystyle{T_{k}f(x)=\int k(x,y) f(y)
d\mu_\kappa(y)}$) defines a bounded mapping of
$L^{p}(\mathbb{R}^{N}, d\mu_\kappa)$ into itself for every $1\leq
p \leq\infty,$ with $$
\|T_{k}\|_{L^{p}(\mathbb{R}^{N}, d\mu_\kappa)\rightarrow
L^{p}(\mathbb{R}^{d},d\mu_\kappa)}\leq
C_{1}^{1-\frac{1}{p}}C_{2}^{\frac{1}{p}}.$$
\end{lemma}

\begin{proof} We note that
 \begin{equation}\label{3.2}
   T_t f -f = \int_0^t \frac{\partial}{\partial s} T_s f ds =
 \int_0^t \Delta_\kappa T_s f ds.\end{equation}
 Fix $p\geq 1$. To estimate the  $L^p(d\mu_\kappa)$-norm of the difference
 (\ref{3.2}), fix $g\in \mathcal{C}_0^\infty(\mathbb{R}^N)$ satisfying
$\|g\|_{p',\kappa}=1$,  where  $p'$ denotes the  H\"older conjugate exponent  of $p$. Integrating (\ref{3.2}) against $g$ and using
 the symmetry of the semi-group $\left(T_t\right)_{t\geq 0}$ yield
 $$ \int \left(T_t f -f\right) g d\mu_\kappa =  \int \left(T_t g -g\right) f d\mu_\kappa = \int_0^t \int
< \nabla_\kappa T_sg, \nabla_\kappa f >d\mu_\kappa ds .$$
 It follows that
  $$\left|\int \left(T_t f -f\right) g d\mu_\kappa \right| \leq \int_0^t
\| \nabla_\kappa T_sg \|_{p',\kappa}\|\nabla_\kappa f \|_{p,\kappa} ds .$$
Thanks to Gaussian estimates (\ref{2.5}) we know that, for any $x\in \mathbb{R}^N$
and $s>0$,
$$ \|\nabla_\kappa h_s(x,.)\|_{1,\kappa} \leq \frac{C}{\sqrt{s}}\int_{\mathbb{R}^N}
\frac{1}{V_\kappa(x, \sqrt{s})} \exp\left(-\frac{cd(x,y)^2}{s}\right)d\mu_\kappa(y).
$$
 A dyadic decomposition on the annulus
 $$B_{2^{j+1} \sqrt s}^G(a) \setminus  B_{2^{j} \sqrt s}^G(a) =
 \left\{ 2^{j} \sqrt s\leq
 d(x,y)\leq  2^{j+1} \sqrt s\right\}, j\in \mathbb{N},$$
 shows that
 $$ \|\nabla_\kappa h_s(x,.)\|_{1,\kappa} \leq Cs^{-1/2} \left[  1 +\sum_{j=0}^\infty\frac{V_\kappa(x, 2^{j+1}\sqrt s)}{V_\kappa(x, 2^{j}\sqrt s) }
 e^{-c 2^{2j}} \right] \leq C s^{-1/2},$$
 because of (\ref{2.2}) and  the doubling volume property (\ref{2.13}). Using Schur's test and taking the supremum over  all  functions $g$ satisfying $\|g\|_{p',\kappa}=1$ give
 (\ref{3.1}). \end{proof}

 Let us now estimate the operator norm of each of the elements of the family
 of operators $\{T_{t^2},\,\,\,0<t\leq r\}$  acting from
 $L^1\left( B_r(a), d\mu_\kappa\right) \rightarrow L^p\left( \mathbb{R}^N, d\mu_\kappa\right)$, for $a\in \mathbb{R}^N$,
  $r>0$ and $p\geq 1$ fixed.
 \begin{proposition}\label{prop3} For any $1\leq p<\infty$, there exists a constant $c>0$ such that for all  $a\in \mathbb{R}^N$, $r>0$, we have
\begin{equation}\label{3.3}
 \| T_{t^2}\|_{L^1\left( B_r(a), d\mu_\kappa\right) \rightarrow
  L^p\left( \mathbb{R}^N, d\mu_\kappa\right)}
 \leq \frac{C}{\left(V_\kappa(a,r)\right)^{\frac{1}{p'}}}
  \left(\frac{r}{t}\right)^{\frac{N+2\gamma}{p'}},\quad 0<t\leq r,\end{equation}
  where $p'$ denotes the H\"older conjugate exponent of $p$.
\end{proposition}
\begin{proof}
An obvious interpolation argument shows that
 \begin{equation}\label{3.4}
    \| T_{t^2}\|_{L^1\left( B_r(a), d\mu_\kappa\right) \rightarrow
  L^p\left( \mathbb{R}^N, d\mu_\kappa\right)}
 \leq \left( \sup\left\{h_{t^2}(x,y), \,\,\,x\in \mathbb{R}^N,\, y \in  B_r(a)\right\}\right)^{
 \frac{1}{p'}}. \end{equation}
 Let $x\in \mathbb{R}^N$, $y \in B_r(a) $ and $0<t\leq r$. Thanks to (\ref{2.4}) and \eqref{2.8}
 $$ h_{t^2}(x,y)  \leq \frac{C}{V_\kappa(y,t)} \leq \frac{C}{V_\kappa(a,r)}
  \left(\frac{r}{t}\right)^{N+2\gamma}.$$
  Hence it follows that:
  \begin{equation}\label{3.5}
 \sup\left\{h_{t^2}(x,y), \,\,\,x\in \mathbb{R}^N,\, y \in  B_r(a)\right\}
  \leq \frac{C}{V_\kappa(a,r)}
  \left(\frac{r}{t}\right)^{N+2\gamma}.\end{equation}
 The proof of (\ref{3.3}) is complete, by applying (\ref{3.4}).
 \end{proof}

Now, we will show how pseudo-Poincar\'e inequality (\ref{3.1})
and the estimation (\ref{3.3}) lead  to  Nash-type inequalities
(\ref{1.4}) and (\ref{1.5}).\\

{\it Proof of Theorem \ref{thm1}. } Fix $1< p<\infty$, $a\in \mathbb{R}^N$, $r>0$ and
   $f\in \mathcal{C}_0^\infty(\mathbb{R}^N) $. Let $0<t \leq r$. Write
$$ \|f\|_{p, \kappa} \leq \| f - T_{t^2} f \|_{p, \kappa} + \| T_{t^2} f \|_{p, \kappa}.$$
Using (\ref{3.1}) and (\ref{3.3}) we obtain
$$ \|f\|_{p, \kappa} \leq Ct \| \nabla_\kappa  f \|_{p, \kappa} +
C \left(\frac{r^{N+2\gamma}}{V_\kappa(a,r)}\right)^{\frac{1}{p'}} t^{-\frac{N+2\gamma}{p'}}
\| f \|_{1, \kappa},$$
where  $p'$ denotes the  H\"older conjugate exponent  of $p$.
Combining with the obvious estimate  $\|f\|_{p, \kappa} \leq t \|f\|_{p, \kappa}
/r$ which is valid for any $t>r$, we deduce that for any $t>0$
$$ \|f\|_{p, \kappa} \leq Ct \left[ \| \nabla_\kappa  f \|_{p, \kappa} +
\frac{1}{r} \|f\|_{p, \kappa}  \right] +
C \left(\frac{r^{N+2\gamma}}{V_\kappa(a,r)}\right)^{\frac{1}{p'}} t^{-\frac{N+2\gamma}{p'}}
\| f \|_{1, \kappa}.$$
Optimizing over $t>0$ yields
$$\|f\|_{p,\kappa}^{1+\frac{p'}{N+2\gamma}}\leq \frac{C r}{(V_\kappa(a,r))^{\frac{1}{N+2\gamma}}}\left(\|\nabla_\kappa f\|_{p,\kappa}+\frac{1}{r}\|f\|_{p,\kappa}\right)\|f\|_{1,\kappa}^{\frac{p'}{N+2\gamma}}.$$

{\it   Proof of Theorem \ref{thm2}.}  Fix $f\in\mathcal{C}_0^\infty(\mathbb{R}^N)$
 and  a real $\lambda>0$.
For any $0<t\leq r$, write
$$ \mu_\kappa \{ |f|\geq \lambda\} \leq
 \mu_\kappa \{| f - T_{t^2} f| \geq \lambda/2\}  +
  \mu_\kappa \{ T_{t^2} |f | \geq \lambda/2\}.$$
  Assume
that $\lambda \geq 4C \|f\|_{1, \kappa}V_\kappa(a,r)^{-1}$ (where $C$ denotes
the constant that appears in  (\ref{3.5})
and pick $t\leq r$ so that
$$
\lambda = 4C t^{-N+2\gamma} \|f\|_{1, \kappa}V_\kappa(a,r)^{-1} r^{N+2\gamma}.$$
Then  it follows from (\ref{3.5}) that
$\| T_{t^2} |f| \|_\infty $ is dominated by $\lambda/4$. Thus
$$ \mu_\kappa \{ |f|\geq \lambda\} \leq
 \mu_\kappa \{| f - T_{t^2} f| \geq \lambda/2\}
 \leq \frac{2}{\lambda}\|f - T_{t^2} f\|_{1, \kappa, }
  \leq 2Ct \| \nabla_\kappa f \|_{1, \kappa}/ \lambda,$$
  where the last inequality is obtained by applying
  (\ref{3.1}). It follows that
  \begin{equation}\label{3.6}
  \mu_\kappa \{ |f|\geq \lambda\} \leq
   C r V_\kappa (a,r)^{-\frac{1}{N + 2 \gamma}}
 \|f\|_{1, \kappa}^{\frac{1}{N + 2 \gamma}}
\| \nabla_\kappa f \|_{1, \kappa}
 \lambda^{-1-\frac{1}{N + 2 \gamma}}.\end{equation}
 On the other hand if $\lambda < 4C \|f\|_{1, \kappa}V_\kappa(a,r)^{-1}$, we simply write
 $$ \mu_\kappa \{ |f|\geq \lambda\} \leq \|f\|_{1, \kappa}
 \lambda ^{-1},$$ which implies
 $$ \mu_\kappa \{ f\geq \lambda\} \leq
   C V_\kappa (a,r)^{-\frac{1}{N+2\gamma }}
 \|f\|_{1, \kappa}^{1+\frac{1}{N+2\gamma }}
 \lambda^{-1-\frac{1}{N+2\gamma }}.$$ Combining
 with (\ref{3.6}) we deduce the weak Nash inequality
 $$\lambda^{1+\frac{1}{N + 2 \gamma}}
  \mu_\kappa \{ |f|\geq \lambda\} \leq
  \frac{Cr}{ V_\kappa (a,r)^{\frac{1}{N + 2 \gamma}} }
  \left( \|  \nabla_\kappa f \|_{1, \kappa}  +
  \frac{1}{r}\|f\|_{1, \kappa}\right)
\|f\|_{1, \kappa}^{\frac{1}{N + 2 \gamma}},\quad\lambda >0.$$

\noindent\textbf{Remark. }
It is easy to see that the considerations of this section immediately generalize to balls $B_r^G(a)$. In particular they lead to the following Sobolev inequality which will be crucial for the applications of the following section.
 \begin{theorem}\label{thm4}
Let $a\in \mathbb{R}^N$ and $r>0$. Then  there exists a constant $C>0$ independent of $a$
 and $r$  such that for all $f\in  \mathcal{C}_0^\infty(B_r^G( a))$:
\begin{equation}\label{3.7}\left(\int_{B_r^G( a)}|f|^{\frac{2(N+2\gamma)}{N+2\gamma-2}}d\mu_\kappa\right)^{\frac{N+2\gamma-2}{N+2\gamma}}\leq  \frac{Cr^2}{\left(V_\kappa(a,r)\right)^{\frac{2}{N+2\gamma}}}
\int_{B_r^G( a)}
\left(|\nabla_\kappa (f)|^{{2}} +\frac{|f|^2}{r^2}\right)d\mu_\kappa.\end{equation}
 \end{theorem}

\section{Mean value inequalities}
In this section we shall derive $L^p$-mean value inequalities using Moser's iterative technique. These inequalities concern
subsolutions of the heat equation  on orbits of balls not necessarily centered on the origin
and are only based on  the Sobolev inequality stated in  (\ref{3.7}).\\

 Let us fix some notations. For $a\in\mathbb{R}^N$, $r>0$, $s\in\mathbb{R}$ and $0<\delta<1$, set
 \begin{eqnarray*}
   Q &=& ]s,s+r^2[\times B_r^G(a)\\
   Q_\delta &=& ]s+\delta r^2,s+r^2[\times B_{(1-\delta)r}^G(a)
 \end{eqnarray*}
 and for a function $u:Q\subset\mathbb{R}\times\mathbb{R}^N\rightarrow \mathbb{N}$, and $p\geq 1$, set
 $$\|u\|^p_{p,Q}=\int_s^{s+r^2}\int_{B_r^G(a)}|u(x,t)|^pd\mu_\kappa(x)dt.$$
 Let us prove some auxiliary results.
 \begin{lemma} \label{lem5.1} Let $u$ be a non-negative parabolic subsolution in $Q$,
 i.e., $u$ satisfies
 $$\left(\frac{\partial}{\partial t}-\Delta_\kappa \right)u\leq 0$$
 in $Q$. Then
  for all $p\geq 2$, $(x,t) \in Q \rightarrow u^p(x,t)$ is also a non-negative subsolution.
 \end{lemma}
 \begin{proof} One has\\

   $\frac{\partial }{\partial t}u^p-\Delta_\kappa u^p$
   \begin{eqnarray*}
      &=& pu^{p-1}\frac{\partial }{\partial t}u-\Delta u^p-2\sum_{\alpha\in \mathcal{R}_+}\kappa(\alpha)\left[pu^{p-1}\frac{<\nabla u(x),\alpha>}{<x,\alpha>}-\frac{u^p(x)-u^p(\sigma_\alpha x)}{<x,\alpha>^2} \right]\\
      &=&  pu^{p-1}\frac{\partial }{\partial t}u-pu^{p-1}\Delta u^p-p(p-1)u^{p-2}|\nabla u|^2\\
      &-&2\sum_{\alpha\in \mathcal{R}_+}\kappa(\alpha)\left[pu^{p-1}\frac{<\nabla u(x),\alpha>}{<x,\alpha>}-\frac{u^p(x)-u^p(\sigma_\alpha x)}{<x,\alpha>^2}\right]\\
      &=&  pu^{p-1}\left(\frac{\partial }{\partial t}-\Delta_\kappa\right)u-p(p-1)|\nabla u|^2\\ &-& 2\sum_{\alpha\in \mathcal{R}_+}\frac{\kappa(\alpha)}{<x,\alpha>^2}\left(pu^{p-1}(x)\left(u(x)-u(\sigma_\alpha x)-u^p(x)+u^p(\sigma_\alpha x)\right)\right).
   \end{eqnarray*}
   Using the fact that $u^p(\sigma_\alpha x)\geq u^p(x)+pu^{p-1}(x)\left(u(\sigma_\alpha x)-u(x)\right)$  ($p$ is greater than $1$),
   we deduce
   $$\left(\frac{\partial }{\partial t}-\Delta_\kappa\right)u^p(x,t)\leq 0,\quad (x,t)\in Q.$$
 \end{proof}
  \begin{proposition}\label{prop4} Let $0<\delta<1$ and let $Q$, $Q_\delta$ and $u$ be as above. Then there exists a positive constant $C$, such that for any $0<\lambda<\eta$ and $p\geq 2$
  \begin{equation}\label{4.1}
     \int_{Q_\eta}u^{2p\theta}d\mu_\kappa dt\leq \frac{Cr^{2(1-\theta)}}
     {\tau^{2(1-\theta)}\left(V_\kappa(a,r)\right)^{\frac{2}{N+2\gamma}}}\left[\int_{Q_\lambda}u^{2p}d\mu_\kappa dt\right]^{\theta}.
   \end{equation}
   where $\theta=\displaystyle 1+\frac{2}{N+2\gamma}$ and $\tau=\eta-\lambda$.
   \end{proposition}
  \begin{proof}
   We observe first that for any non-negative function $\phi\in  \mathcal{C}_0^\infty(B_r^G(a))$, we have
  \begin{equation}\label{4.2}
    \int \left(\phi\frac{\partial}{\partial t}u+<\nabla_\kappa\phi,\nabla_\kappa u>\right)d\mu_\kappa=\int \phi(\frac{\partial}{\partial t}- \Delta_\kappa) u d\mu_\kappa \leq 0.
  \end{equation}
 Set
  $$\Gamma_\kappa(\phi,u)=<\nabla\phi ,\nabla u>+\sum_{\alpha\in\mathcal{R}_+}\kappa(\alpha)\frac{\left(\phi(x)-\phi(\sigma_\alpha x)\right)\left(u(x)-u(\sigma_\alpha x)\right)}{<\alpha,x>^2}.$$
  With $\phi=\psi^2 u$, we obtain
  \begin{eqnarray*}\label{5.21}
    \Gamma_\kappa(\psi^2 u,u)&=&\left[2\psi\nabla\psi u+\psi^2\nabla u\right]\nabla u\\ &+&\sum_{\alpha\in\mathcal{R}_+}\kappa(\alpha)\frac{\left((\psi^2 u)(x)-(\psi^2 u)(\sigma_\alpha x)\right)\left(u(x)-u(\sigma_\alpha x)\right)}{<\alpha,x>^2}.
  \end{eqnarray*}
Since \begin{eqnarray}\label{4.3}
         \nonumber \left[2\psi\nabla\psi u+\psi^2\nabla u\right]\nabla u
           &=& 2\psi\nabla\psi u \nabla u+\psi^2(\nabla u)^2\\
          \nonumber &=& \left(u\nabla\psi+\psi \nabla u\right)^2-u^2(\nabla\psi)^2 \\
          \nonumber &=& \left(u\nabla\psi+\psi\nabla u\right)^2-u^2(\nabla\psi)^2\\
          &=& \left(\nabla(\psi u)\right)^2 -u^2(x)(\nabla\psi)^2(x).
        \end{eqnarray}
        Otherwise
        $\left((\psi^2 u)(x)-(\psi^2 u)(\sigma_\alpha x)\right)\left(u(x)-u(\sigma_\alpha x)\right)$

        \begin{eqnarray}\label{4.4}
         \nonumber  &=& \left[\psi^2(x)u^2(x)+\psi^2(\sigma_\alpha  x)u^2(\sigma_\alpha x)\right]\\  \nonumber & -&\left[\psi^2(x)u(x)u(\sigma_\alpha  x)-\psi^2(\sigma_\alpha  x)u(x)u(\sigma_\alpha  x)\right] \\
           &=& \left(\psi u(x) -\psi u(\sigma_\alpha x)\right)^2-u(x)u(\sigma_\alpha x)\left(\psi(x)-\psi(\sigma_\alpha x)\right)^2 .
        \end{eqnarray}
        (\ref{4.3}) and (\ref{4.4}) lead to
$$\displaystyle\Gamma_\kappa(\psi^2 u,u)=\Gamma_\kappa(\psi u)-u^2(x)(\nabla\psi)^2(x)-u(x)\sum_{\alpha\in\mathcal{R}_+}\kappa(\alpha)u(\sigma_\alpha x)\frac{\left(\psi(x)-\psi(\sigma_\alpha x)\right)^2}{<\alpha,x>^2},$$
where we use the notation $\Gamma_\kappa(v)= \Gamma_\kappa(v,v)$.
  Applying (\ref{4.2}) to $\psi^2 u$, where $\psi\in \mathcal{C}_0^\infty(B_r^G( a))$, it follows from the previous computation that\\
	
	$\displaystyle \int_{B_r^G(a)}\left(\psi^2 u\frac{\partial u}{\partial t}+\Gamma_\kappa(\psi u)\right)d\mu_\kappa$
  \begin{eqnarray*}&\leq &\int_{B_r^G(a)}u^2(x)(\nabla\psi)^2(x)d\mu_\kappa(x)\\
	&+& \int_{B_r^G(a)}u(x)\sum_{\alpha\in\mathcal{R}_+}\kappa(\alpha)u(\sigma_\alpha x)
	\frac{\left(\psi(x)-\psi(\sigma_\alpha x)\right)^2}{<\alpha,x>^2}d\mu_\kappa(x). \end{eqnarray*}
  Assuming $\psi$ invariant under the action of $G$ we deduce then that
  \begin{equation}\label{4.5}
    \int_{B_r^G(a)}\left(\psi^2 u\frac{\partial u}{\partial t}+\Gamma_\kappa(\psi u)\right)d\mu_\kappa\leq \|\nabla\psi\|_\infty^2\int_{\textrm{supp}\, \psi}u^2d\mu_\kappa.
  \end{equation}
  Let $\chi$ denotes a non-negative smooth function of the time variable. We have
  \begin{eqnarray}\label{4.6}
    \frac{\partial}{\partial t} \int_{B_r^G(a)}\left(\chi\psi u\right)^2 d\mu_\kappa &=& 2  \int_{B_r^G(a)}\left(\frac{d\chi}{dt}\chi\psi^2 u^2+\frac{\partial u}{\partial t}u\psi^2\chi^2\right)d\mu_\kappa \nonumber\\
     &\leq& 2\chi\|\chi'\|_\infty \int_{B_r^G(a)}\psi^2 u^2 d\mu_\kappa+2\chi^2\int_{B_r^G(a)}\psi^2 u \frac{\partial u}{\partial t}d\mu_\kappa.
  \end{eqnarray}

  Combining (\ref{4.5}) and (\ref{4.6}) we deduce that
  \begin{eqnarray*}\frac{\partial}{\partial t} \int_{B_r^G(a)}\left(\chi\psi u\right)^2 d\mu_\kappa+\chi^2\int_{B_r^G(a)} \Gamma_\kappa(\psi u)d\mu_\kappa&\leq & 2\chi\|\chi'\|_\infty\|\psi\|^2_\infty\int_{\textrm{supp}\, \psi}u^2d\mu_\kappa\\ &+&2\chi^2\|\nabla\psi\|^2_\infty\int_{\textrm{supp}\, \psi}u^2d\mu_\kappa. \end{eqnarray*}
  Hence
  \begin{eqnarray}\nonumber
   \frac{\partial}{\partial t} \int_{B_r^G(a)}\left(\chi\psi u\right)^2d\mu_\kappa&+&\chi^2\int_{B_r^G(a)}\Gamma_\kappa(\psi u)d\mu_\kappa\\
	&\leq& \label{4.7}2\chi\left[\chi\|\nabla\psi\|^2_\infty+\|\chi'\|_\infty\|\psi\|^2_\infty\right]\int_{\textrm{supp}\, \psi}u^2d\mu_\kappa.
  \end{eqnarray}
  Let $\psi$ satisfying
  $$\left\{\begin{array}{cccc}
  0\leq \psi  \leq  |G| &&\\
  \textrm{supp }\, \psi &\subset & B_{(1-\lambda)r}^G(a) \\
   \psi = |G| & \textrm{on} & B_{(1-\eta)r}^G(a) \\
  |\nabla \psi| & \leq &\displaystyle \frac{|G|}{\tau r}.
  \end{array}\right.$$
  To construct such $\psi$ it suffices to choose a function $\psi_a$ such that
   $$\left\{\begin{array}{ccc}
             0\leq \psi_a \leq 1 & &\\
             \textrm{supp}(\psi_a)& \subset& B_{(1-\lambda)r}(a) \\
           \psi_a=1,& \textrm{on} & B_{(1-\eta)r}(a)\\
           |\nabla\psi_a|\leq (\tau r)^{-1} &&
           \end{array}\right.
   $$
    and choose $$\psi(x)=\displaystyle\sum_{\sigma\in G}\psi_a(\sigma x),\quad x\in \mathbb{R}^N.$$
  Fix $s\in \mathbb{R}$ and $\chi$ such that
  $$\left\{\begin{array}{ccc}
      0 \leq \chi  \leq  1& &\\
      \chi=0   &\textrm{on} & ]-\infty,s+\lambda r^2[ \\
     \chi=1  &\textrm{on} & ]s+\eta r^2,+\infty[ \\
      |\chi'| \; \leq  \displaystyle\frac{1}{\tau r^2}&&.
    \end{array}\right.
  $$
Integrating (\ref{4.7}) over $]s,t[$ with $t\in ]s+\lambda r^2,s+r^2[$, we obtain
  \begin{eqnarray}\label{4.8}
    \sup_{s+\eta r^2<t<s+r^2}\left\{\int_{B_{(1-\eta)r}^G(a)}u^2d\mu_\kappa\right\} &+&\int_{s+\eta r^2}^{s+r^2}\int_{B_{(1-\eta)r}^G(a)}\Gamma_\kappa(u)d\mu_\kappa dt\\ & \leq & \frac{4|G|}{\tau^2r^2}\int_{s+\lambda r^2}^{s+r^2}\int_{B_{(1-\lambda)r}G(a)}u^2d\mu_\kappa dt.\nonumber
  \end{eqnarray}
  Thanks to H\"older's inequality, we have
  $$\int|f|^{2\theta}d\mu_\kappa\leq \left(\int|f|^{\frac{2(N+2\gamma)}{N+2\gamma-2}}d\mu_\kappa\right)^{\frac{N+2\gamma-2}{N+2\gamma}}
  \left(\int|f|^{2}d\mu_\kappa\right)^{\frac{2}{N+2\gamma}}.$$
  Combining with Sobolev's inequality (\ref{3.7}) gives
  $$\int|f|^{2\theta}d\mu_\kappa\leq\frac{Cr^2}{\left(V_\kappa(a,r)\right)^{\frac{1}{N+2\gamma}}}\left(\int \left(\Gamma_\kappa(f)+\frac{|f|^2}{r^2}\right)d\mu_\kappa\right)\left(\int|f|^{2}d\mu_\kappa\right)^{\frac{2}{N+2\gamma}}$$
  for all $f\in \mathcal{C}_0^\infty(B_r^G(a))$. The above inequality gives for a subsolution $u$
   \begin{eqnarray}\label{4.9}
   \int_{Q_\eta}u^{2\theta}d\mu_\kappa dt&\leq & \frac{C(1-\eta)^2r^2}{\left(V_\kappa(a,(1-\eta)r)\right)^{\frac{2}{N+2\gamma}}}\left[\int_{Q_\eta}\left(\Gamma_\kappa(u)+\frac{u^2}{r^2}\right)d\mu_\kappa dt\right] \nonumber \\ &\times&\sup_{s+\eta r^2<t<s+r^2}\left(\int_{B_G(a,(1-\eta)r)}u^2d\mu_\kappa\right)^{\frac{2}{N+2\gamma}}.
   \end{eqnarray}
   Combining (\ref{4.8}) with (\ref{4.9}) we deduce
   \begin{equation}\label{4.10}
     \int_{Q_\eta}u^{2\theta}d\mu_\kappa dt\leq\frac{C(1-\eta)^2r^2}{\left(V_\kappa(a,(1-\eta)r)\right)^{\frac{2}{N+2\gamma}}}\left[\frac{4|G|+1}{\tau^2 r^2}\int_{Q_\lambda}u^{2}d\mu_\kappa dt\right]^{\theta}.
   \end{equation}
   Using Lemma \ref{lem5.1} and applying (\ref{4.10}) to $u^p$ completes the proof of Proposition \ref{prop4}.
   \end{proof}

   Our next step is to prove the following $L^p$ mean value inequality.
   \begin{theorem}\label{thm5}  Let $0<\delta<1$ and let $Q$, $Q_\delta$
   and $u$ be as above. Then there exists a  constant $C >0$ such that for
   $p\geq 2$ and  any non-negative subsolution in $Q$,
 \begin{equation*}\label{}
    \sup_{Q_{\delta}}u\leq C\left(\frac{\delta^{-(N+2\gamma+2)}}{r^2V_\kappa(a,r)}\right)^{\frac{1}{p}}\|u\|_{p,Q}.
 \end{equation*}
  \end{theorem}
  \begin{proof} We resume the notation of the proof of Proposition \ref{prop4}.
   Set for $i\in\mathbb{N}$
   $$\lambda_0=0,\quad \lambda_i=\delta\sum_{j=1}^i2^{-j}, i\geq 1.$$
    Applying Proposition \ref{prop4} with $p=p_i=p\theta^i$, $\lambda=\lambda_i$, $\eta=\lambda_{i+1}$, then $\tau_i=\lambda_{i+1}-\lambda_i=\delta2^{-1-i}$,
   $$\int_{Q_{\lambda_{i+1}}}u^{2p_{i+1}}d\mu_\kappa dt\leq \frac{r^2}{\left(V_\kappa(a,r)\right)^{\frac{2}{N+2\gamma}}}\left[\frac{(4|G|^2+1)4^{i+1}}{ r^2\delta^2}\int_{Q_{\lambda_{i}}}u^{2pii}d\mu_\kappa dt\right]^{\theta}.$$ Hence\\
	
	$\displaystyle \left[\int_{Q_{\lambda_{i+1}}}u^{2p_{i+1}}d\mu_\kappa dt\right]^{\frac{1}{p_{i+1}}}$
   \begin{eqnarray*}&\leq & \left(\frac{Cr^2}{\left(V_\kappa(a,r)\right)^{\frac{2}{N+2\gamma}}}\right)^{\frac{1}{p_{i+1}}}\left(\frac{(4|G|^2+1)4^{i+1}}{ r^2\delta^2}\right)^{\frac{1}{p_{i}}}\left[\int_{Q_{\lambda_{i}}}u^{2pi}d\mu_\kappa dt\right]^{\frac{1}{p_{i}}}\\
	&\leq &\left[4^{C_2(i)}\left(\frac{r^2}{\left(V_\kappa(a,r)\right)^{\frac{2}{N+2\gamma}}}\right)^{C_{1}(i+1)-1}
   (Cr\delta)^{-2{C_1(i)}}\int_Q u^2d\mu_\kappa dt\right]^{\frac{1}{p}},\end{eqnarray*}
where $$C_1(i)=\sum_{j=0}^{i}\theta^{-j};\quad C_2(i)=\sum_{j=0}^i(j+1)\theta^{-j}.$$
   Observe that $\lambda_i\rightarrow \delta$ as $i\rightarrow +\infty$ $$\displaystyle \sum_{j=1}^\infty\theta^{-j}=\frac{N+2\gamma}{2}.$$
   and $$\lim_{q\rightarrow+\infty}\|f\|_{q,\kappa}=\|f\|_{q,\infty}.$$
   Thus, letting $i\rightarrow +\infty$, we obtain
   $$\sup_{Q_{\delta}}u\leq C(N,\gamma)\left(\frac{\delta^{-(N+2\gamma+2)}}{r^2V_\kappa(a,r)}\right)^{\frac{1}{p}}\|u\|_{p,Q},$$
where $C(N,\gamma)=\displaystyle \left[4^{\left(\frac{N+2\gamma+2}{2}\right)^2}C^{\frac{N+2\gamma+2}{2}}\right]^{\frac{1}{p}}$.
   This completes  the proof of Theorem \ref{thm5}.
   \end{proof}

Finally we  extend Theorem \ref{thm5} to  $0<p<2$.
  \begin{corollary} Fix $0<p<2$ and let $\delta\in ]0,1[$. Then, any non-negative $u$ such that
  $$\left(\frac{\partial}{\partial t}-\Delta_\kappa\right)u\leq 0\quad \textrm{in}\; Q=]s,s+r^2[\times B_r^G(a);$$
  $a\in \mathbb{R}^N$, $r>0$, satisfies
  \begin{equation*}
    \sup_{Q_\delta}u\leq C\delta^{-\frac{N+2\gamma+2}{p}}\left(r^2V_\kappa(a,r)\right)^{-\frac{1}{p}}\|u\|_{p,Q}
  \end{equation*}
  where the constant $C>0$ is independent of $u,\delta,a,r$ and $s$.
  \end{corollary}
  \begin{proof}
  Let $0<p<2$. Fix $0<\rho<\displaystyle\frac{1}{2}$ and set $\tau=\displaystyle\frac{\rho}{4}$. Theorem \ref{thm5} yields
  $$\sup_{Q_\rho}u\leq C\tau^{-\frac{N+2\gamma+2}{2}}\left(r^2V_\kappa(a,r)\right)^{-\frac{1}{2}}\|u\|_{2,Q_{\rho-\tau}}.$$
  Now, as $\displaystyle\|u\|_{2,Q_{\rho-\tau}}\leq \|u\|_{p,Q}^{\frac{p}{2}}\left(\sup_{Q_{\rho-\tau}}u\right)^{1-\frac{p}{2}},$ we get
  \begin{equation}\label{4.11}
    \sup_{Q_\rho}u\leq C\tau^{-\frac{N+2\gamma+2}{2}}\left(r^2V_\kappa(a,r)\right)^{-\frac{1}{2}}\|u\|_{p,Q}^{\frac{p}{2}}\left(\sup_{Q_{\rho-\tau}}u\right)^{1-\frac{p}{2}}.
  \end{equation}
   Let $0<\delta<\displaystyle \frac{1}{2}$, $\rho_0=\delta$ and set for $i=0,1\ldots$ $\tau_{i+1}=\displaystyle \frac{\rho_i}{4}$ and $\rho_{i+1}=\rho_i-\tau_{i+1}$. Applying (\ref{4.11}) for each $i$  yields
   $$\sup_{Q_{\rho_{i-1}}}u\leq C4^{\frac{N+2\gamma+2}{2}i}\delta^{-\frac{N+2\gamma+2}{2}}\left(r^2V_\kappa(a,r)\right)^{-\frac{1}{2}}\|u\|_{p,Q}^{\frac{p}{2}}
   \left(\sup_{Q_{\rho_i}}u\right)^{1-\frac{p}{2}}.$$
  Integrating gives
   $$\sup_{Q_{\delta}}u\leq C_i\, \left(\delta^{-\frac{N+2\gamma+2}{2}}\left(r^2V_\kappa(a,r)\right)^{-\frac{1}{2}}
   \|u\|_{p,Q}^{\frac{p}{2}}\right)^{\displaystyle\sum_{j=0}^{i-1}(1-\frac{p}{2})^j}
   \left(\sup_{Q_{\delta}}u\right)^{(1-\frac{p}{2})^i},$$
   where $C_i=C^{\frac{N+2\gamma+2}{2}\displaystyle\sum_{j=0}^{i-1}(j+1)(1-\frac{p}{2})^j}$. When  $i$ tends to infinity, this yields
   $$\sup_{Q_{\delta}}u\leq C^{\frac{N+2\gamma+2}{p^2}}\left(\delta^{-\frac{N+2\gamma+2}{2}}\left(r^2V_\kappa(a,r)\right)^{-\frac{1}{2}}\right)^{\frac{p}{2}}
   \|u\|_{p,Q},$$
which implies the desired inequality.
\end{proof}
\section*{Data availability statement} My manuscript has no associated data.

%\begin{acknowledgements}
%\end{acknowledgements}

\section*{Conflict of interest} The authors declare that they have no conflict of interest.


\begin{thebibliography}{00}
 \bibitem{AAP}  Adhikari, S.,  Anoop, V. P., Parui, S.: Existence of an extremal of Dunkl-type Sobolev inequality and Stein-Weiss inequality for D-Riesz potential. Complex Anal. Oper. Theory \textbf{15}, 28 (2021). https://doi.org/10.1007/s11785-020-01068-1
\bibitem{An} Anker,  J.-P.: An introduction to Dunkl theory and its analytic aspects, In Analytic, algebraic and
geometric aspects of differential equations, Trends Math., pages 3-58. Birkh\"{a}user/Springer, Cham
(2017).
\bibitem{An2}  Anker, J.-P. Dziuba\'{n}ski, J., Hejna, A.: Harmonic functions, Conjugate harmonic functions and the Hardy space $H^1$ in the rational Dunkl setting. J. Fourier Anal. Appl. \textbf{25}, no. 5, 2356--2418(2019). https://doi.org/10.1007/s00041-019-09666-0
    \bibitem{BCLS} Bakry, D.,  Coulhon, T., M. Ledoux, Saloff-Coste, L.: Sobolev inequalities in deguise. Indiana Univ. Math. J., \textbf{44}, no 4, 1033--1074 (1995).
     \bibitem{CKS}  Carlen, E. A., Kusuoka, S., Stroock, D. W.: Upper bounds for symmetric Markov transition
functions. Ann. inst. H. Poincar\'{e} Non Lin\'{e}aire \textbf{23}(1987), 245--287 (1987).
   \bibitem{deJ} de Jeu, M.F.E.: The Dunkl transform. Invent. Math. \textbf{113}, 147-162 (1993. https://doi.org/10.1007/BF01244305
\bibitem{Du} Dunkl, C.F.: Differential-difference operators associated to reflection groups. Trans. Amer. Math.
\textbf{311}, no. 1, 167--183 (1989). https://doi.org/10.1090/S0002-9947-1989-0951883-8
\bibitem{Du2} Dunkl, C.F.: Integral kernels with reflection group invariance. Canad. J. Math. \textbf{43}, 1213-1227(1991). https://doi:10.4153/CJM-1991-069-8
\bibitem{Na} Nash, J.: Continuity of solutions of parabolic and elliptic equations. Amer. J. Math., \textbf{80}, 931--954(1958). https://doi.org/10.2307/2372841
\bibitem{R1} R\"{o}sler, M.: Dunkl operators: theory and applications. In: Orthogonal polynomials and special functions
(Leuven, 2002), volume \textbf{1817} of Lecture Notes in Math., pages 93-135. Springer, Berlin (2003).
\bibitem{RV} R\"{o}sler, M.,  Voit, M.: Dunkl theory, convolution algebras, and related Markov processes. In: Harmonic
and stochastic analysis of Dunkl processes, P. Graczyk, M. R\"osler, M. Yor (eds.), 1-112, Travaux
en cours 71, Hermann, Paris (2008).
\bibitem{Sa0}  Saloff-Coste, L.: A note on Poincar\'{e}, Sobolev, and Harnack inequalities. Internat. Math. Res. Notices, no. 2, 27-38(1992). https://doi.org/10.1155/S1073792892000047
\bibitem{Sa1} Saloff-Coste, L.: Aspects of Sobolev-type inequalities. London Mathematical Society Lecture Note Series, \textbf{289}. Cambridge University Press, Cambridge, 2002. x+190 pp.
 \bibitem{V2} Velicu, A.: Sobolev-Type Inequalities for Dunkl Operators,
  J. Funct. Anal., \textbf{279}, no. 7, 37 pp (2020). https://doi.org/10.1016/j.jfa.2020.108695

 \end{thebibliography}


\end{document}